\documentclass[a4paper, 12pt]{amsart}

\usepackage[a4paper,margin=25mm]{geometry}  
\usepackage{amsmath}
\usepackage{amssymb}
\usepackage{graphicx}
\usepackage{amsthm}
\usepackage{caption}
\usepackage[initials]{amsrefs}

\BibSpec{article}{%
  +{}{\PrintAuthors} {author}
  +{,}{ \textit} {title}
  +{,}{ } {journal}
  +{}{ \textbf} {volume}
  +{}{ \parenthesize} {date} 
  +{,}{ } {pages}
  +{,}{ } {note}
  +{.}{} {transition}
  +{}{ } {review}
}

\makeatletter
    
    \@addtoreset{equation}{section}
  \makeatother
\theoremstyle{plain}
\newtheorem{thm}{Theorem}[section]
\newtheorem*{thm*}{Theorem A}
\newtheorem*{thm**}{Theorem B}

\newtheorem{cor}[thm]{Corollary}
\newtheorem{lem}[thm]{Lemma}
\theoremstyle{definition}

\newtheorem{df}[thm]{Definition}

\newcommand{\R}{\mathbb{R}}
\newcommand{\Z}{\mathbb{Z}}

\newcommand{\GL}{\mathop{\mathrm{GL}}\nolimits}

\newcommand{\ad}{\mathop{\mathrm{ad}}\nolimits}

\newcommand{\Der}{\mathop{\mathrm{Der}}\nolimits}

\newcommand{\dsum}{\displaystyle\sum}

\newcommand{\ddelta}{\mathop{\scriptstyle{\Delta}}\nolimits}
\newcommand{\relmiddle}[1]{\mathrel{}\middle#1\mathrel{}}

\usepackage[unicode,psdextra]{hyperref}

\begin{document}
\title{Infinitely many left-symmetric structures on nilpotent Lie algebras}
\author{Naoki Kato}
\date{}
\maketitle

\begin{abstract}
Dekimpe and Ongenae constructed infinitely many pairwise non-isomorphic complete left-symmetric structures on ${\R}^n$ for $n\geq 6$.
In this paper, we construct a family of complete left-symmetric structures on the cotangent Lie algebra $T^*\mathfrak{g}$ of a certain $n$-dimensional almost abelian nilpotent Lie algebra $\mathfrak{g}$ and give a condition under which two left-symmetric structures in this family are isomorphic.
As a consequence of this result, we obtain infinitely many pairwise non-isomorphic left-symmetric structures on $T^{*}\mathfrak{g}$.
As an application of this construction, we also obtain infinitely many symplectic structures on $T^{*}\mathfrak{g}$ which are pairwise non-symplectomorphic up to homothety.
\end{abstract}

\section{Introduction}
Let $G$ be a connected Lie group.
A left-invariant flat and torsion-free affine connection on $G$  is called a left-invariant affine structure on $G$.
In the case where $G$ is simply connected, it is known that the Lie group $G$ has a left-invariant affine structure if and only if its Lie algebra $\mathfrak{g}$ has a left-symmetric structure.

In \cite{Mil77}, Milnor conjectured that any solvable Lie group admits a complete left-invariant affine structure.
Benoist \cite{Ben95} and Burde and Grunewald \cite{BurGru95} proved that some filiform Lie algebras have no left-symmetric structures.
Such Lie algebras provide counterexamples to Milnor's conjecture.
It is therefore an important problem to determine which Lie algebras admit a left-symmetric structure.
It is known that  $2$ and $3$-step nilpotent Lie algebras, as well as nilpotent Lie algebras of dimension $n\leq 6$, admit left-symmetric structures (\cite{Sch74}, \cite{Bur06}).

On the other hand, for a given Lie algebra $\mathfrak{g}$, it is an important problem to classify left-symmetric structures on $\mathfrak{g}$ or determine how many left-symmetric structures $\mathfrak{g}$ admits.
Left-symmetric structures on ${\R}$, ${\R}^{2}$, and ${\R}^{3}$ have been classified (see \cite{GozRem03}, \cite{Kui53}, and \cite{NagYag74}).
Nilpotent left-symmetric structures on $3$ and $4$-dimensional Lie algebras were classified in \cite{Kim86}.
The classification of simple left-symmetric structures on $3$-dimensional solvable Lie algebras and of simple complete left-symmetric structures on $4$-dimensional solvable Lie algebras was given in \cite{Bur98}.

According to classification results, the number of isomorphism classes of left-symmetric structures are finite for certain Lie algebras.
On the other hand, it is known that some Lie algebras admit infinitely many isomorphism classes of left-symmetric structures.
Dekimpe and Ongenae \cite{DekOng00} constructed infinitely many complete left-symmetric structures on ${\R}^n$ for $n\geq 6$.
The author \cite{Kat24} constructed infinitely many incomplete left-symmetric structures on ${\R}^n$ for $n\geq 8$ by using the left-symmetric structures constructed by Dekimpe and Ongenae.

A closed non-degenerate skew-symmetric bilinear form on $\mathfrak{g}$ is called a symplectic structure on $\mathfrak{g}$.
It is known that a symplectic structure induces a left-symmetric structure (see Definition \ref{df4-1} or \cite{Chu74}).
It is also known that a cosymplectic structure induces a left-symmetric structure \cite{BouMan24}.
Symplectic structures on $4$-dimensional solvable Lie algebras were classified in \cite{Ova06} and those on $6$-dimensional nilpotent Lie algebras were classified in \cite{KhaGozMed04}.
Cosymplectic structures on $3$ and $5$-dimensional solvable Lie algebras were classified in \cite{BouMan24}.
Although symplectic structures themselves have been widely studied, the connection between symplectic structures and algebraic or geometric properties of induced left-symmetric structures has not been studied extensively.

Equivalence classes of symplectic structures up to symplectomorphism and homothety  have been studied in \cite{Cas22} and \cite{CasTam23}.
Castellanos Moscoso \cite{Cas22} proved that, for certain almost abelian Lie algebras, the set of equivalence classes of symplectic structures is a single point. 
Castellanos Moscoso and Tamaru \cite{CasTam23} proved that the set of equivalence classes of symplectic structures on $\mathfrak{h}_3\oplus {\R}^{n-3}$ is a single point, where $\mathfrak{h}_3$ is the $3$-dimensional Heisenberg Lie algebra.

For a given symplectic structure, the induced left-symmetric structure is uniquely determined by the equivalence class of the symplectic structure.
Hence, one can prove that two symplectic structures are not symplectomorphic up to homothety by showing that the corresponding left-symmetric structures are not isomorphic.

In this paper, for a certain almost abelian nilpotent Lie algebra $\mathfrak{g}$, we construct infinitely many complete left-symmetric structures on the cotangent Lie algebra $T^{*}\mathfrak{g}$  of $\mathfrak{g}$ and give a condition under which these left-symmetric structures are isomorphic (Theorem \ref{thm3-1}).
As a consequence, for any $n\geq 2$, we obtain infinitely many pairwise non-isomorphic complete left-symmetric structures on the $n$-step nilpotent Lie algebra $T^{*}\mathfrak{g}$ (Corollary \ref{cor3-1}).

As an application of Theorem \ref{thm3-1}, we obtain infinitely many symplectic structures on $T^{*}\mathfrak{g}$ which are pairwise non-symplectomorphic up to homothety (Corollary \ref{cor4-1}).
Consequently, the set of equivalence classes of symplectic structures on $T^{*}\mathfrak{g}$ contains a continuous family.

 \section{Preliminary}
\subsection{Left-symmetric structures}
 Let $\mathfrak{g}$ be a finite-dimensional  Lie algebra over ${\R}$.
\begin{df}
A bilinear product $\ddelta \colon  \mathfrak{g}\times \mathfrak{g}\to\mathfrak{g}$ is said to be a left-symmetric structure on $\mathfrak{g}$ if $\ddelta$ satisfies 
\begin{align}
x\ddelta(y\ddelta z)-(x \ddelta y)\ddelta z&= y\ddelta (x\ddelta z)-(y\ddelta x)\ddelta z \quad \text{and} \label{df1-1}\\
[x,y]&=x\ddelta y-y\ddelta x\label{df1-2}
\end{align}
for any $x,y,z\in \mathfrak{g}$.

Two left-symmetric structures $\ddelta$ on $\mathfrak{g}$ and $\ddelta'$ on  $\mathfrak{g}'$ are said to be isomorphic if there exists an isomorphism of Lie algebras $\phi\colon\mathfrak{g}\to\mathfrak{g}'$ such that $\phi(x\ddelta y)=\phi(x)\ddelta' \phi(y)$ for any $x,y\in\mathfrak{g}$.
 
\end{df}
Let $(x,y,z)=x\ddelta(y\ddelta z)-(x\ddelta y)\ddelta z$ be the associator of $\ddelta$.
Then the equation \eqref{df1-1} is equivalent to the equation $(x,y,z)=(y,x,z)$.

Let $\ell\colon\mathfrak{g}\to\mathfrak{gl(g)}$ and $r\colon \mathfrak{g}\to\mathfrak{gl(g)}$ be left and right translation maps, which are defined by 
\begin{align*}
\ell(x)&=\ell_x\colon y\mapsto x\ddelta y \quad \text{ and}\\
r(x)&=r_x\colon y\mapsto y\ddelta x,
\end{align*}
respectively.
Then the equation \eqref{df1-2} is written as $\ad_x=\ell_x-r_x$, where $\ad\colon\mathfrak{g}\to\mathfrak{gl(g)}$ is the adjoint representation of $\mathfrak{g}$.
Under \eqref{df1-2}, the equation \eqref{df1-1} holds if and only if the left translation map $\ell$ is a homomorphism of Lie algebras.
 
\begin{df}
A left-symmetric structure $\ddelta $ on $\mathfrak{g}$ is said to be complete if $r_x\in\mathfrak{gl(g)}$ is nilpotent for any $x\in\mathfrak{g}$.
\end{df}
See \cite{Kim86} for other equivalent definitions of completeness.
We remark that a left-symmetric structure is complete if and only if the corresponding left-invariant affine structure on the simply connected Lie group is complete (see \cite{Hel79} and \cite{Kim86}).

\subsection{Almost abelian Lie algebras}

\begin{df}
A Lie algebra $\mathfrak{g}$ is said to be almost abelian if $\mathfrak{g}$ has a codimension one abelian ideal.
\end{df}

Let $D$ be a derivation of a  Lie algebra ${\mathfrak{h}}$.
Define a Lie bracket on ${\R}\oplus\mathfrak{h} $ by
$$
[a+x,b+y]=aD(y)-bD(x)+ [x,y]_{\mathfrak{h}}\quad \text{ for } a,b\in{\R} \text{ and } x,y\in\mathfrak{h}, 
$$
where $[\, , \,]_{\mathfrak{h}}$ is the Lie bracket of $\mathfrak{h}$.
We denote the Lie algebra ${\R}\oplus\mathfrak{h} $ with the Lie bracket defined above by ${\R}\ltimes_D\mathfrak{h}$ and call it the semi-direct sum of ${\R}$ and $\mathfrak{h}$ by $D$.

Let $\mathfrak{g}$ be an $(n+1)$-dimensional almost abelian Lie algebra.
Then $\mathfrak{g}$ is isomorphic to ${\R}\ltimes_D{\R}^n$ for some $D\in\Der({\R}^{n})=\mathfrak{gl}({\R}^{n})$.
If $D$ is nilpotent, then ${\R}\ltimes_D{\R}^n$ is a nilpotent Lie algebra.
Moreover, if the nilpotency index of $D$ is $k$, then ${\R}\ltimes_D{\R}^n$ is a $k$-step nilpotent Lie algebra.

Two semi-direct sums ${\R}\ltimes_{D_1}{\R}^{n}$ and ${\R}\ltimes_{D_2}{\R}^{n}$ are isomorphic if and only if there exist $P\in\GL({\R}^{n})$ and $\lambda\in{\R}\setminus \{0\}$ such that $D_1=\lambda P^{-1}D_2P$ (see \cite{Ave22}).

\subsection{Cotangent Lie algebras}
Let $\mathfrak{g}$ be an $n$-dimensional Lie algebra.
Define a Lie bracket on $\mathfrak{g}\oplus\mathfrak{g}^*$  by
$$
[x+\phi,y+\psi] = [x,y]_{\mathfrak{g}}+\ad^*_x\psi-\ad^*_y \phi \quad \text{ for  }x, y\in\mathfrak{g} \text{ and }\phi,\psi \in\mathfrak{g}^*,
$$
where $\ad^*\colon\mathfrak{g}\to\mathfrak{gl(g^*)}$ is the coadjoint representation of $\mathfrak{g}$ defined by $\ad^*_x\phi=-\phi\circ\ad_x$.
We denote the $2n$-dimensional Lie algebra $(\mathfrak{g}\oplus\mathfrak{g^*}, [\, , \, ])$ by $T^*\mathfrak{g}$ and call it the cotangent Lie algebra of $\mathfrak{g}$.
For properties of cotangent Lie algebras, see \cite{Ova16}.

\section{Main result}
For a Lie algebra $\mathfrak{g}$, we define the lower central series of $\mathfrak{g}$ recursively by 
$$
\mathfrak{g}^{(0)}=\mathfrak{g} \quad \text{and } \mathfrak{g}^{(k)}=[\mathfrak{g}^{(k-1)}, \mathfrak{g}].
$$

For an integer $n\geq 2$, let $J_n(0)\in\mathfrak{gl}({\R}^{n})$ be the Jordan block of size $n$ with eigenvalue $0$ and $\mathfrak{g}={\R}\ltimes_{J_n(0)}{\R}^{n}$ the semi-direct sum of ${\R}$ and ${\R}^{n}$ by $J_n(0)$, which is almost abelian and $n$-step nilpotent.

Let $\{t\}$ and $\{e_1,\ldots ,e_n\}$ be the standard basis of ${\R}$ and ${\R}^{n}$, respectively.
Then $\{t,e_1,\ldots ,e_n\}$ is a basis of $\mathfrak{g}$ and the Lie bracket of $\mathfrak{g}$ is given by
$$
[t,e_{i+1}]=e_{i} \quad \text{for }i=1,\ldots ,n-1.
$$

We consider the cotangent Lie algebra $T^*\mathfrak{g}$.
Let $\{z\}$ and $\{e^{1},\ldots ,e^{n}\}$ be the dual basis of $\{t\}$ and $\{e_1,\ldots ,e_n\}$, respectively.
For $i=1,\ldots, n$, we put $f_i=e^{{n-i+1}}$.
Then $\mathcal{E}=\{z,e_1, f_1, \ldots ,e_n, f_n, t\}$ is a basis of $T^{*}\mathfrak{g}$.
 The Lie bracket of $T^{*}\mathfrak{g}$ is given by 
 \begin{align*}
 [t,e_{i+1}]&=e_i,\\
 [t,f_{i+1}]&=-f_i, \text{ and}\\
 [e_{i+1}, f_{n-i+1}]&=z
 \end{align*}
for $i=1,\ldots ,n-1$.
Then we have
\begin{align}
(T^{*}\mathfrak{g})^{{(k)}}&=\langle z, e_1,\ldots ,e_{n-k}, f_1, \ldots , f_{n-k}\rangle \quad  \label{eq3-1}\text{ and } \\
c(T^{*}\mathfrak{g})&=(T^{*}\mathfrak{g})^{(n-1)}=\langle z, e_1, f_1\rangle \label{eq3-2}
\end{align}
for $ k=1,\ldots ,n-1$.
If $n\geq 3$, then we have
$$
[(T^{*}\mathfrak{g})^{(1)}, (T^{*}\mathfrak{g})^{(1)}]=\langle z\rangle.
$$

Let $\alpha$ and  $\beta$ be real numbers which satisfy
\begin{align}
k(\alpha-1)+1 &\not= 0 \quad \text{ for any }k\in{\Z}_{\geqq 1},\label{cond3-1}\\
k(\beta+1)-1 &\not= 0\quad \text{ for any }k\in{\Z}_{\geqq 1}, \text{ and}\label{cond3-2}\\
n(\alpha-1)(\beta+1)-\alpha+\beta+2 &\not= 0.\label{cond3-3}
\end{align}
For $i=1,\ldots, n$, we put
\begin{align}
\alpha_i&=\dfrac{i(\alpha-1)+1}{(\alpha-1)(i-1)+1},  \label{eq3-3}\\
\beta_i&=\dfrac{i(\beta+1)-1 }{-(\beta +1)(i-1)+1}, \quad \text{and}  \label{eq3-4}\\
\gamma_i&=\dfrac{\alpha_i\beta_{n-i}-\beta_{n-i}}{2\alpha_i\beta_{n-i}+\alpha_i-\beta_{n-i}}=\dfrac{(\alpha-1)\{(n-i)(\beta+1)-1\}}{n(\alpha-1)(\beta+1)-(\alpha-1)+\beta+1}. \label{eq3-5}
\end{align}
We remark that the sequences $\{\alpha_i\}$ and $\{\beta_i\}$ are the solutions to the recurrence relations of 
\begin{align}
\alpha_{i+1}&=2-\dfrac{1}{\alpha_i}, \alpha_1=\alpha \quad \text{and}  \label{eq3-6}\\
\beta_{i+1}&=-2-\dfrac{1}{\beta_i}, \beta_1=\beta, \label{eq3-7}
\end{align} 
respectively.
The sequence $\{\gamma_i\}$ satisfies the equations
\begin{align}
\gamma_i+\beta_{n-i}\gamma_{i+1}=0 \quad \text{and} \label{eq3-8}\\
(\alpha_{i+1}-1)(\gamma_i-1)-(\beta_{n-i}+1)\gamma_{i+1}=0. \label{eq3-9}
\end{align}

Define a bilinear product $\ddelta_{(\alpha, \beta)}\colon T^{*}\mathfrak{g}\times T^{*}\mathfrak{g}\to T^{*}\mathfrak{g}$ by
\begin{align}
\begin{aligned}
t\ddelta_{(\alpha, \beta)} e_{i+1}&=\alpha_i e_i, \\
e_{i+1}\ddelta_{(\alpha, \beta)} t&=(-1+\alpha_i)e_i,\\ 
t\ddelta_{(\alpha, \beta)} f_{i+1}&=\beta_i f_i,\\
f_{i+1}\ddelta_{(\alpha, \beta)}t&=(1+\beta_i)f_i, \\
e_{i+1}\ddelta_{(\alpha, \beta)} f_{n-i+1}&=\gamma_i z,\\
f_{n-i+1}\ddelta_{(\alpha, \beta)} e_{i+1}&=(-1+\gamma_i)z,
\end{aligned}\label{df}
\end{align}
and the other products are zero, where $i=1,\ldots ,n-1$.

\begin{lem}
The bilinear product $\ddelta_{(\alpha, \beta)}$ is a complete left-symmetric structure on $T^*\mathfrak{g}$.
\end{lem}

\begin{proof}
By the definition of $\ddelta_{(\alpha,\beta)}$, equation \eqref{df1-2} holds.
To verify \eqref{df1-1}, it suffices to prove that
\begin{align}
[\ell_t, \ell_{e_{i+1}}]&=\ell_{e_i}  &&(i=1,\ldots ,n-1), \label{lem31-1}\\
[\ell_t,\ell_{f_{i+1}}]&=-\ell_{f_i}   &&(i=1,\ldots, n-1),\label{lem31-2}\\
[\ell_{e_{i+1}},\ell_{f_{n-i+1}}] &=\ell_z   &&(i=1,\ldots, n-1), \text{ and} \label{lem31-3}\\
[\ell_{e_i},\ell_{f_j}]&=0  &&(i+j\not= n+2) \label{lem31-4}.
\end{align}

Put $e'_0=t, e'_i=e_i, e'_{n+i}=f_i,$ and $e'_{2n+1}=z$ for $i=1,\ldots, n$.
For $i,j=0,\ldots, 2n+1$, let $E_{ij}$ denote the matrix unit.
Then, by the definition of $\ddelta_{(\alpha,\beta)}$, the representation matrices of $\ell_t, \ell_{e_{i+1}}, \ell_{f_{i+1}},$ and $ \ell_{f_{n-i+1}}$ with respect to the basis $\{e'_0,\ldots, e'_{2n+1}\}$ are given by
\begin{align*}
\ell_t&=\dsum_{k=1}^{n-1}(\alpha_k E_{k, k+1}+\beta_k E_{n+k, n+k+1}) ,\\
\ell_{e_{i+1}}&=(-1+\alpha_i)E_{i, 0}+\gamma_i E_{2n+1, 2n-i+1}, \\
\ell_{f_{i+1}}&=(1+\beta_i)E_{n+i, 0}+(-1+\gamma_{n-i})E_{2n+1, n-i+1}, \quad \text{ and}\\
\ell_{f_{n-i+1}}&=(1+\beta_{n-i})E_{2n-i, 0}+(-1+\gamma_i)E_{2n+1, i+1}
\end{align*}
for $i=1,\ldots ,n-1.$
Then, for $i=2,\ldots , n-1$, we have the following:
\begin{align*}
\ell_t\ell_{e_{i+1}}&=\alpha_{i-1}(-1+\alpha_i)E_{i-1, 0}.\\
\ell_{e_{i+1}}\ell_t& =\gamma_{i}\beta_{n-i+1}E_{2n+1, 2n-i+2}.\\
\ell_{t}\ell_{e_2}&=\ell_{e_2}\ell_t=0.\\
\ell_t\ell_{f_{i+1}}&=\beta_{i-1}(1+\beta_i)E_{n+i-1,0}. \\
\ell_{f_{i+1}}\ell_t&=(-1+\gamma_{n-i})\alpha_{n-i+1}E_{2n+1, n-i+2}.\\
\ell_t\ell_{{f_2}}&=\ell_{f_2}\ell_t=0.\\
\ell_{e_{i+1}}\ell_{f_{n-i+2}} &=\gamma_{i}(1+\beta_{n-i+1})E_{2n+1, 0}.\\
\ell_{f_{n-i+2}}\ell_{e_{i+1}}&=(-1+\gamma_{i-1})(-1+\alpha_{i})E_{2n+1, 0}.\\
\ell_{e_i}\ell_{f_j}&=\ell_{f_j}\ell_{e_i}=0 \quad (i+j\not= n+3).
\end{align*}
Hence, by \eqref{eq3-6}, \eqref{eq3-7}, \eqref{eq3-8}, and \eqref{eq3-9}, we obtain the following:
\begin{align*}
[\ell_t, \ell_{e_{i+1}}]
&=\alpha_{i-1}(-1+\alpha_i)E_{i-1, 0}-\gamma_{i}\beta_{n-i+1}E_{2n+1, 2n-i+2}\\
&=(\alpha_{i-1}-1)E_{i-1, 0}+\gamma_{i-1}E_{2n+1, 2n-i+2}\\
&=\ell_{e_i} \quad (i=2,\ldots ,n-1).\\
[\ell_t,\ell_2]
&=0=\ell_{e_1}.\\
[\ell_t,\ell_{f_{i+1}}]
&=\beta_{i-1}(1+\beta_i)E_{n+i-1, 0}-(-1+\gamma_{n-i})\alpha_{n-i+1}E_{2n+1, n-i+2} \\
&=(-1-\beta_{i-1})E_{n+i-1, 0}-\{\gamma_{n-i}-1+(\beta_i+1)\gamma_{n-i+1}\}E_{2n+1, n-i+2}\\
&=-(1+\beta_{i-1})E_{n+i-1,0}-(-1+\gamma_{n-i+1})E_{2n+1, n-i+2}\\
&=-\ell_{f_i} \quad (i=2,\ldots, n-1)\\
[\ell_t,\ell_{f_2}]
&=0=-\ell_{f_1}.\\
[\ell_{e_{i+1}}, \ell_{f_{n-i+1}}]&=0=\ell_z.\\
[\ell_{e_{i+1}},\ell_{f_{n-i+2}}]
&=\gamma_{i}(1+\beta_{n-i+1})E_{2n+1, 0}-(-1+\gamma_{i-1})(-1+\alpha_{i})E_{2n+1, 0}=0.\\
[\ell_{e_i}, \ell_{f_j}]
&=0 \quad (i+j\not= n+2, n+3).
\end{align*}
Therefore the equations \eqref{lem31-1}, \eqref{lem31-2}, \eqref{lem31-3}, and \eqref{lem31-4} hold. 

Moreover, by the definition of $\ddelta_{(\alpha,\beta)}$, for any $x\in T^{*}\mathfrak{g}$, $r_x$ is nilpotent.
Hence the left-symmetric structure $\ddelta_{(\alpha,\beta)}$ is complete.
\end{proof}

\begin{thm}\label{thm3-1}
Let $\{\alpha_i\}, \{\beta_i\}$ and $\{\alpha'_i\}, \{\beta_i'\}$ be the sequences defined by \eqref{eq3-3} and $\eqref{eq3-4}$ with initial values $\alpha_1=\alpha, \beta_1=\beta$ and $\alpha_1'=\alpha',  \beta_1'=\beta'$, respectively.
Suppose that $\alpha_i, \alpha_i'\not= \dfrac{1}{2}$ and $\beta_i,\beta_i'\not=-\dfrac{1}{2}$ for any $i=1,\ldots n$, and that  $\alpha_{n-1}, \alpha'_{n-1}\not= 0, 1$ and  $\beta_{n-1}, \beta'_{n-1}\not= 0, -1$.

Then two left-symmetric structures $\ddelta=\ddelta_{(\alpha,\beta)}$ and $\ddelta'=\ddelta_{(\alpha', \beta')}$ on $T^{*}\mathfrak{g}$ are isomorphic if and only if either of the following conditions holds:
\begin{itemize}
\item[{\rm(i)}] $\alpha=\alpha'$ and $\beta=\beta'$.
\item[{\rm (ii)}] $\alpha_i=-\beta_i'$ and $\beta_i=-\alpha_i'$ for each $i=1,\ldots, n-1$.
\end{itemize}
\end{thm}

We remark that the sequences $\{\alpha_i\}$ and $\{\beta_i\}$ satisfy
\begin{align*}
|\alpha_{i+1}|
&=\left| 2-\dfrac{1}{\alpha_i}\right|
\geq 2-\dfrac{1}{|\alpha_{i}|}
\quad \text{and}\\
|\beta_{i+1}|
&=\left| 2+\dfrac{1}{\beta_i}\right|
\geq 2-\dfrac{1}{|\beta_i|},
\end{align*}
respectively.
By these inequalities,  $|\alpha|>1$ and $|\beta|>1$ imply $|\alpha_i|>1$ and $|\beta_i|>1$ for any $i$, respectively.
Hence, for example, the sequences $\{\alpha_i\}$ and $\{\beta_i\}$ with 
$$
(\alpha_1, \beta_1)\in A=\{(\alpha,\beta)\mid \alpha>1, \beta>1, n(\alpha-1)(\beta-1)-\alpha+\beta+2\not= 0\}
$$
satisfy the conditions \eqref{cond3-1}, \eqref{cond3-2}, and \eqref{cond3-3} and the assumption of Theorem \ref{thm3-1}.
Moreover, any two distinct elements of $A$ do not satisfy the condition (ii). 
Therefore, any two distinct left-symmetric structures in $\{\ddelta_{(\alpha,\beta)}\mid (\alpha,\beta)\in A\}$ are non-isomorphic.

\begin{cor}\label{cor3-1}
For any integer $n\geq 2$, the $n$-step nilpotent Lie algebra $T^{*}\mathfrak{g}$ admits uncountably many pairwise non-isomorphic complete left-symmetric structures.
\end{cor}

\begin{proof}[{\bf Proof of Theorem \ref{thm3-1}}]

Suppose that there exists an isomorphism $\phi\colon (T^{*}\mathfrak{g},\ddelta_{(\alpha,\beta)})\to (T^{*}\mathfrak{g},\ddelta_{(\alpha',\beta')}).$
Since $\phi$ is an automorphism of the Lie algebra $T^{*}\mathfrak{g}$, $\phi$ preserves $(T^{*}\mathfrak{g})^{(k)}$ for $k=1, \ldots, n-1$.
The matrix of $\phi$ with respect to the basis $\mathcal{E}=\{z,e_1, f_1, \ldots ,e_n, f_n, t\}$ is as follows:
\begin{align*}
\phi(z)&=z_0' z+c_1e_1+d_1f_1.\\
\phi(e_j)&=z_j z+\dsum_{i=1}^{j}a_{ij}e_i+\dsum_{i=1}^{j}b_{ij}f_i  \quad (j=1,\ldots, n-1).\\
\phi(f_j)&=z_j' z+\dsum_{i=1}^{j}c_{ij}e_i +\dsum_{i=1}^{j}d_{ij} f_i \quad (j=1,\ldots, n-1).\\
\phi(e_n)&=z_nz+\dsum_{i=1}^{n}a_{in}e_i+\dsum_{i=1}^{n}b_{in}f_i +t_n t.\\ 
\phi(f_n)&=z_n'z+\dsum_{i=1}^{n}c_{in}e_i +\dsum_{i=1}^{n}d_{in} f_i+t_n' t.\\
\phi(t)&=z_0 z+\dsum_{i=1}^{n}a_ie_i+\dsum_{i=1}^{n} b_i f_i +t_0 t.
\end{align*}
Since the matrix of $\phi$ is regular, the following holds:
\begin{align}
&\begin{pmatrix} z_0'&z_1&z_1'\\c_{1}& a_{11}&c_{11}\\d_1&b_{11}&d_{11}\end{pmatrix} \text{ is regular.} \label{reg3-1} \\
&\begin{pmatrix} a_{kk}&c_{kk}\\ b_{kk}&d_{kk}\end{pmatrix} \text{ is regular for }  k=2,\ldots, n-1.\label{reg3-2} \\
&\begin{pmatrix} a_{nn}&c_{nn}&a_n\\b_{nn}&d_{nn}&b_n\\ t_n&t_n'&t_0\end{pmatrix} \text{ is regular}.\label{reg3-3}
\end{align}

\begin{lem}\label{lem3-2}
For $i=2,\ldots, n$, $a_i=0, b_i=0, t_n=0, t'_n=0$, and $t_0\not= 0$.
\end{lem}

\begin{proof}
Since $t\ddelta t=0$, we have
\begin{align*}
0
&=\phi(t)\ddelta'\phi(t)\\
&=(z_0 z+\dsum_{i=1}^{n}a_ie_i+\dsum_{i=1}^{n} b_i f_i +t_0 t)\ddelta'(z_0 z+\dsum_{j=1}^{n}a_je_j+\dsum_{j=1}^{n} b_j f_j +t_0 t)\\
&=Kz+t_0\dsum_{k=1}^{n} (a_k t\ddelta' e_k+b_kt\ddelta f_k+a_ke_k\ddelta t+b_k f_k\ddelta t)\\
&=Kz+t_0\dsum_{k=2}^{n}(a_k \alpha'_{k-1}e_{k-1}+b_k \beta'_{k-1}f_{k-1}+a_k(-1+\alpha'_{k-1})e_{k-1}+b_k(1+\beta'_{k-1})f_{k-1})\\
&=Kz+t_0\dsum_{k=2}^{n}(a_k(2\alpha'_{k-1}-1)e_{k-1}+b_k(2\beta'_{k-1}+1)f_{k-1}),
\end{align*}
where $K\in{\R}$ is defined by
$$
\dsum_{i,j=1}^{n}(a_ia_j e_i\ddelta' e_j+a_ib_je_i\ddelta' f_j+ b_ia_j f_i\ddelta' e_j+b_ib_jf_i\ddelta f_j)=Kz.
$$
Hence, by the assumption of Theorem \ref{thm3-1}, we have
\begin{equation}
\begin{split}
&\text{either }t_0=0 \\
&\text{or } a_k=0 \text{ and }b_k=0  \quad (k=2,\ldots, n).
\end{split}\label{lem34-1}
\end{equation}

By similar calculations for $e_n\ddelta e_n=0, f_n\ddelta f_n=0$, and $e_n\ddelta f_n=0$, we obtain the following three equations:
\begin{align*}
0
&=\phi(e_n)\ddelta' \phi(e_n)\\
&=Lz+t_n\dsum_{k=2}^{n}( a_{kn}(2\alpha'_{k-1}-1)e_{k-1}+b_{kn}(2\beta'_{k-1}+1)f_{k-1}).\\
0
&=\phi(f_n)\ddelta'\phi(f_n)\\
&=Mz+t'_n\dsum_{k-2}^{n}(c_{kn}(2\alpha'_{k-1}-1)e_{k-1}+d_{kn}(2\beta'_{k-1}+1)f_{k-1}).\\
0
&=\phi(e_n)\ddelta'\phi(f_n)\\
&=Nz+\dsum_{k=2}^{n}\left\{(t_nc_{kn}\alpha'_{k-1}+t_n'a_{kn}(\alpha'_{k-1}-1))e_{k-1}+(t_nd_{kn}\beta'_{k-1}+t_n'b_{kn}(\beta'_{k-1}+1)f_{k-1})\right\},
\end{align*}
where $L, M, N\in{\R}$.
Therefore, by the assumption of Theorem \ref{thm3-1},  the following  holds: 
\begin{align}
&\begin{aligned}
&\text{either } t_n=0\\
&\text{or } a_{kn}=0 \text{ and } b_{kn}=0  \quad (k=2,\ldots, n).
\end{aligned}\label{lem34-2}
\\
&\begin{aligned}
&\text{either }t_n'=0\\
&\text{or } c_{kn}=0 \text{ and } d_{kn}=0  \quad  (k=2,\ldots, n).
\end{aligned}\label{lem34-3}
\\
&\begin{aligned}
&\text{both }t_nc_{kn}\alpha_{k-1}'+t_n'a_{kn}(\alpha_{k-1}'-1)=0  \quad (k=2,\ldots, n)\\
&\text{and }t_n d_{kn}\beta'_{k-1}+t_n'b_{kn}(\beta_{k-1}'+1)=0  \quad (k=2,\ldots, n).
\end{aligned}\label{lem34-4}
\end{align}

By the condition \eqref{lem34-1}, \eqref{lem34-2}, and \eqref{lem34-3}, the regular matrix $C=\begin{pmatrix} a_{nn}&c_{nn}&a_n\\b_{nn}&d_{nn}&b_n\\ t_n&t_n'&t_0\end{pmatrix}$ must be one of the following forms:
\begin{itemize}
\item[(A)] $\begin{pmatrix} 0&c_{nn}&a_n\\0&d_{nn}&b_n\\ t_n&0&0\end{pmatrix}$  with  $t_n\not= 0$.
\item[(B)] $\begin{pmatrix} a_{nn}&0&a_n\\b_{nn}&0&b_n\\ 0&t_n'&0\end{pmatrix}$ with $t_n'\not= 0$.
\item[(C)] $\begin{pmatrix} a_{nn}&c_{nn}&0\\b_{nn}&d_{nn}&0\\ 0&0&t_0\end{pmatrix}$  with $t_0\not= 0$. 
\end{itemize}
If $C$ is of the form (A), then we have $c_{nn}=0$ and $d_{nn}=0$ by \eqref{lem34-4} and the assumption of Theorem \ref{thm3-1}.
This contradicts the regularity of  $C$.
Similarly, if $C$ is of the form (B), we obtain a contradiction with the regularity of $C$.

Thus we conclude that $t_n=0, t_n'=0, t_0\not= 0, a_k=0$, and $b_k=0$ for $k=2,\ldots, n$ by \eqref{lem34-1}.
\end{proof}

By Lemma \ref{lem3-2}, $\phi(t)=z_0 z+a_1 e_1+b_1 f_1+t_0 t$.
\begin{lem}\label{lem3-3}
\mbox{}
\begin{itemize}
\item[(i)] For $i=1,\ldots, n-1, z_i=z_i'=0$.
\item[(ii)] For $l=1,\ldots ,n-1$ and $k=1,\ldots ,l$, 
$$
\begin{pmatrix} a_{kl}&c_{kl}\\b_{kl} &d_{kl}\end{pmatrix}=t_0^{{n-l}}\begin{pmatrix}a_{k+n-l \, n} & (-1)^{{n-l}}c_{k+n-l \, n}\\ (-1)^{n-l}b_{k+n-l \, n} & d_{k+n-l \, n}\end{pmatrix}.
$$
\end{itemize}
\end{lem}

\begin{proof}
For $l=1,\ldots ,n-1$, since $[t,e_{l+1}]=e_{l}$, we have
\begin{align*}
\phi(e_{l})
&=[\phi(t),\phi(e_{l+1})]\\
&=\left[z_0z+a_1e_1+b_1f_1+t_0 t, z_{l+1}z+\dsum_{k=1}^{l+1} a_{k\, l+1} e_k+\dsum_{k=1}^{l+1}b_{k\, l+1}f_k\right]\\
&=t_0\dsum_{k=2}^{l+1}a_{k\, l+1}e_{k-1}-t_0\dsum_{k=2}^{l+1}b_{k\, l+1} f_{k-1}\\
&=t_0\dsum_{k=1}^{l}a_{k+1\, l+1}e_{k}-t_0\dsum_{k=1}^{l}b_{k+1\, l+1} f_{k}.
\end{align*}
Hence  we obtain 
$$
z_{l}=0, \,  t_0a_{k+1\, l+1}= a_{kl}, \text{ and } -t_0b_{k+1\, l+1}=b_{k\, l}
$$
for $k=1,\ldots, l$.
By an inductive argument, it follows that 
$$
a_{kl}=t_0^{n-l}a_{k+n-l\, n}  \text{ and }  b_{kl}=(-t_0)^{{n-l}}b_{k+n-l\, n}.
$$

By an analogous calculation for $[t,f_{l+1}]=-f_{l}$, we obtain
$$
z_{l}'=0,\, c_{kl}=(-t_0)^{n-l}c_{k+n-l\, n}, \text{ and }  d_{k l}=t_0^{{n-l}}d_{k+n-l\, n}
$$
for $k=1,\ldots, l$.

\end{proof}

\begin{lem}\label{lem3-4}
Either of the following holds:
\begin{itemize}
\item[(i)] For each $l=1,\ldots, n-1$, $\alpha_l=\alpha_l'$ and $\beta_l=\beta_l'$.
\item[(ii)] For each $l=1,\ldots, n-1$, $\alpha_l=-\beta_l'$ and $\beta_l=-\alpha_l'$.
\end{itemize}
\end{lem}

\begin{proof}
For $l=1,\ldots, n-1$, by the equations $t\ddelta e_{l+1}=\alpha_l e_l$ and $t\ddelta f_{l+1}=\beta_l f_l$, we have
\begin{align*}
\alpha_l\phi(e_l)
&=(z_0z+a_1e_1+b_1f_1+t_0t)\ddelta' \left(z_{l+1}z+\dsum_{k=1}^{l+1} a_{k\, l+1}e_k+\dsum_{k=1}^{l+1} b_{k\, l+1}f_k\right)\\
&=t_0\dsum_{k=2}^{l+1}a_{k\, l+1}\alpha_{k-1}' e_{k-1} +t_0\dsum_{k=2}^{l+1} b_{k\, l+1} \beta_{k-1}'f_{k-1}\\
&=t_0\dsum_{k=1}^{l}a_{k+1\, l+1}\alpha_{k}' e_{k} +t_0\dsum_{k=1}^{l} b_{k+1\, l+1} \beta_{k}'f_{k} \quad \text{ and}\\
\beta_l\phi(f_l)
&=(z_0z+a_1e_1+b_1f_1+t_0t)\ddelta' \left(z_{l+1}'z+\dsum_{k=1}^{l+1} c_{k\, l+1}e_k+\dsum_{k=1}^{l+1} d_{k\, l+1}f_k\right)\\
&=t_0\dsum_{k=2}^{l+1}c_{k\, l+1}\alpha_{k-1}' e_{k-1} +t_0\dsum_{k=2}^{l+1} d_{k\, l+1} \beta_{k-1}'f_{k-1}\\
&=t_0\dsum_{k=1}^{l}c_{k+1\, l+1}\alpha_{k}' e_{k} +t_0\dsum_{k=1}^{l} d_{k+1\, l+1} \beta_{k}'f_{k} . 
\end{align*}

Hence we obtain
$$
\begin{pmatrix}\alpha_l a_{ll}&\beta_l c_{ll}\\\alpha_l b_{ll}& \beta_l d_{ll}\end{pmatrix}=t_0\begin{pmatrix}\alpha_l'a_{l+1\, l+1}& \alpha_l' c_{l+1\, l+1}\\ \beta_l'b_{l+1\, l+1}&\beta_l' d_{l+1\, l+1}\end{pmatrix}.
$$
By applying Lemma \ref{lem3-3}, we have
$$
t_0^{n-l}\begin{pmatrix} \alpha_l a_{nn}& (-1)^{n-l} \beta_l c_{nn} \\(-1)^{n-l}\alpha_l b_{nn} & \beta_l d_{nn}\end{pmatrix}
=t_0^{n-l}\begin{pmatrix} \alpha_l' a_{nn}&(-1)^{n-l-1} \alpha_l' c_{nn}\\ (-1)^{n-l-1}\beta_l'b_{nn}   & \beta_l' d_{nn}\end{pmatrix}.
$$
Since $t_0\not= 0$ by Lemma \ref{lem3-2},  $\alpha_l, \alpha_l', \beta_l$, and $\beta_l'$ satisfy 
$$
\begin{pmatrix}(\alpha_l-\alpha_l')a_{nn} & (\beta_l+\alpha_l')c_{nn} \\ (\alpha_l+\beta_l') b_{nn} & (\beta_l-\beta_l')d_{nn}\end{pmatrix}=O.
$$

Since $C=\begin{pmatrix} a_{nn}&c_{nn}&0\\b_{nn}&d_{nn}&0\\ 0&0&t_0\end{pmatrix}$ is regular, the matrix $\begin{pmatrix} a_{nn}&c_{nn}\\b_{nn}&d_{nn}\end{pmatrix}$ is regular.
Hence either one of the following holds:
\begin{itemize}
\item $a_{nn}\not= 0$ and $c_{nn}\not= 0$.
\item $a_{nn}\not=0$ and $c_{nn}= 0$.
\item $a_{nn}=0$ and $c_{nn}\not= 0$.
\end{itemize}
In the case where $a_{nn}\not=0$ and $c_{nn}\not=0$, we obtain $\alpha_l=\alpha_l'$ and $\beta_l=-\alpha_l'$.
If $b_{nn}\not=0$, then $\alpha_l=-\beta_l'$ and we obtain $\beta_l=\beta_l'$.
If $b_{nn}=0$, then $d_{nn}\not=0$ and we obtain $\beta_l=\beta_l'$.
In the case where $a_{nn}\not=0$ and $c_{nn}=0$, $d_{nn}\not= 0$ and we obtain $\alpha_l=\alpha_l'$ and $\beta_l=\beta_l'$.
In the case where $a_{nn}=0$ and $c_{nn}\not= 0$,  $b_{nn}\not=0$ and we obtain $\alpha_l=-\beta_l'$ and $\beta_l=-\alpha_l'$.

\end{proof}

Therefore either condition (i) or (ii) of Theorem \ref{thm3-1} holds.

Conversely, suppose that either condition (i) or condition (ii) of Theorem \ref{thm3-1} holds.
If condition (i) holds, then the identity map is an isomorphism from $(T^{*}\mathfrak{g}, \ddelta)$ to $(T^{*}\mathfrak{g}, \ddelta')$.

If condition (ii) holds, we define a linear isomorphism $\phi\colon T^{*}\mathfrak{g}\to T^{*}\mathfrak{g}$ by 
\begin{align*}
\phi(z)&=(-1)^{n} z,\\
\phi(e_j)&=(-1)^{n-j+1}f_j,\\
\phi(f_j)&=(-1)^{n-j}e_j, \quad \text{ and}\\
\phi(t)&=t
\end{align*}
for $j=1,\ldots, n$.
We show that $\phi$ is an isomorphism from $(T^{*}\mathfrak{g},\ddelta)$ to $(T^{*}\mathfrak{g},\ddelta')$.

\begin{lem}\label{lem3-5} 
$\gamma_i$ and $\gamma_i'$ satisfy 
$$
\gamma_{n-i}'+\gamma_{i}=1 \quad \text{for } i=1,\ldots, n-1
 $$
 \end{lem}

\begin{proof}
Since $\alpha_i=-\beta_i'$ and $\beta_i=-\alpha_i'$, we have
$$
\gamma_i
=\dfrac{(\alpha_i-1)\beta_{n-i}}{2\alpha_i\beta_{n-i}+\alpha_i-\beta_{n-i}}
=\dfrac{(\beta_i'+1)\alpha'_{n-i}}{2\beta_i'\alpha'_{n-i}-\beta'_i+\alpha'_{n-i}}.
$$

Therefore,
\begin{align*}
\gamma_{n-i}'+\gamma_i
&=\dfrac{(\alpha'_{n-i}-1)\beta'_i}{2\alpha'_{n-i}\beta_i'+\alpha_{n-i}'-\beta_i'}+\dfrac{(\beta_i'+1)\alpha'_{n-i}}{2\beta_i'\alpha'_{n-i}-\beta'_i+\alpha'_{n-i}}\\
&=\dfrac{2\alpha_{n-i}'\beta_i'-\beta_i'+\alpha_{n-i}'}{2\beta_i'\alpha'_{n-i}-\beta'_i+\alpha'_{n-i}}\\
&=1.
\end{align*}

\end{proof}

For $j=1,\ldots, n-1$, we have
\begin{align*}
\phi(t)\ddelta' \phi(e_{j+1})
&=(-1)^{{n-j}}t\ddelta' f_{j+1}
=(-1)^{{n-j}}\beta_j' f_{j}
=(-1)^{n-j+1}\alpha_j f_{j}\\
&=\phi(\alpha_j e_j)
=\phi(t\ddelta e_{j+1})\\
\phi(t)\ddelta' \phi(f_{j+1})
&=(-1)^{{n-j-1}}t\ddelta' e_{j+1}
=(-1)^{{n-j-1}}\alpha_j' e_{j}
=(-1)^{n-j}\beta_j e_{j}\\
&=\phi(\beta_j f_j)
=\phi(t\ddelta f_{j+1})\\
\phi(e_{j+1})\ddelta'\phi(t)
&=(-1)^{{n-j}}f_{j+1}\ddelta' t
=(-1)^{{n-j}}(1+\beta_j') f_{j}\\
&=(-1)^{n-j+1}(-1+\alpha_j) f_{j}
=\phi((-1+\alpha_j)e_j)
=\phi( e_{j+1}\ddelta t)\\
\phi(f_{j+1})\ddelta'\phi(t)
&=(-1)^{{n-j-1}} e_{j+1}\ddelta' t
=(-1)^{{n-j-1}}(-1+\alpha_j') e_{j}\\
&=(-1)^{n-j}(1+\beta_j) e_{j}
=\phi((1+\beta_j)f_j)
=\phi(f_{j+1}\ddelta t)\\
\phi(e_{j+1})\ddelta' \phi(f_{n-j+1})
&=(-1)^{{n-j}}f_{j+1}\ddelta' (-1)^{j-1}e_{n-j+1}
=(-1)^{n-1}(-1+\gamma_{n-j}')z\\
&=(-1)^{n}\gamma_i z
=\phi(\gamma_i z)
=\phi(e_{j+1}\ddelta f_{n-j+1})\\
\phi(f_{n-j+1})\ddelta' \phi(e_{j+1})
&=(-1)^{j-1}e_{n-j+1}\ddelta' (-1)^{n-j}f_{j+1}
=(-1)^{n-1}(\gamma_{n-j}')z\\
&=(-1)^{n-1}(1-\gamma_i) z
=\phi((-1+\gamma_i) z)
=\phi(f_{n-j+1}\ddelta e_{j+1}).
\end{align*}
Therefore $\phi$ is an isomorphism from $(T^{*}\mathfrak{g}, \ddelta)$ to $(T^{*}\mathfrak{g}, \ddelta')$.

\end{proof}


\section{An application for symplectic structures}
Let $\mathfrak{g}$ be a $2n$-dimensional Lie algebra.

\begin{df}\label{df4-1}
A non-degenerate skew-symmetric bilinear form $\omega$ on $\mathfrak{g}$ is said to be a symplectic structure on $\mathfrak{g}$ if $\omega$ satisfies $d\omega=0$, that is,  $\omega$ satisfies
$$
\omega(x,[y,z])+\omega(y,[z,x])+\omega(z,[x,y])=0
$$
for any $x,y,z\in\mathfrak{g}$.

Let $\omega$ and $\omega'$ be symplectic structures on $\mathfrak{g}$.
Two symplectic structures $\omega$ and $\omega'$ are said to be symplectomorphic up to homothety (or equivalent up to automorphism and scale) if there exists an automorphism $\phi$ of the Lie algebra $\mathfrak{g}$ and $c\not=0$ such that $\phi^{*}\omega'=c\omega$.
\end{df}
Some geometric properties of the set of equivalence classes of symplectic structures have been studied by Castellanos Moscoso and Tamaru in \cite{CasTam23}.

It is known that a symplectic structure induces a left-symmetric structure.
\begin{df}\label{df4-2}
Let $\omega$ be a symplectic structure on $\mathfrak{g}$.
The left-symmetric structure $\ddelta_\omega$ on $\mathfrak{g}$ defined by
\begin{equation}
\omega(x\ddelta_\omega y,z)=-\omega(y,[x, z])
\end{equation}
is called the left-symmetric structure induced by $\omega$.
\end{df}
For properties of left-symmetric structures induced by symplectic structures, see \cite{Chu74} and  \cite{MedRev91}.

\begin{lem}\label{lem4-1}
If $\omega$ and $\omega'$ are symplectomorphic up to homothety, then the left-symmetric structures $\ddelta_{\omega}$ and $\ddelta_{\omega'}$ are isomorphic.
\end{lem}
\begin{proof}
Let $\phi$ be an automorphism of $\mathfrak{g}$ such that $\phi^{*}\omega'=c \omega$ for some $c\not= 0$.
Then $\phi$ satisfies 
\begin{align*}
\omega'(\phi(x\ddelta_{\omega}y), \phi(z))
&=\phi^{*}\omega'(x\ddelta_{\omega} y, z)
=c\omega(x\ddelta_{\omega} y, z)
=-c\omega(y, [x,z])\\
&=-\phi^{*}\omega'(y,[x,z])
=-\omega'(\phi(y), [\phi(x),\phi(z)])\\
&=\omega'(\phi(x)\ddelta_{\omega'} \phi(y),\phi(z)).
\end{align*}
Hence $\phi$ satisfies $\phi(x\ddelta_{\omega}y)=\phi(x)\ddelta_{\omega'}\phi(y)$.
\end{proof}

For a Lie algebra $\mathfrak{g}$, the cotangent Lie algebra $T^{*}\mathfrak{g}$ admits an $\ad$-invariant non-degenerate symmetric bilinear form (see \cite{Ova16}).
Hence, if $T^{*}\mathfrak{g}$ admits a symplectic structure, then $\mathfrak{g}$ must be nilpotent (see \cite{BajBenMed07}).

Let $\mathfrak{g}={\R}\ltimes_{J_n(0)}{\R}^{n}$ be the semi-direct sum of ${\R}$ and ${\R}^{n}$ by $J_n(0)$ and consider the cotangent Lie algebra $T^{*}\mathfrak{g}$ of $\mathfrak{g}$ with the basis $\mathcal{E}=\{z,e_1, f_1, \ldots ,e_n, f_n, t\}$.
Denote its dual basis by $\{z^{*}, e_{1}^*, f_{1}^*, \ldots ,e_{n}^*, f_{n}^*, t^{*}\}$.

\begin{lem}\label{lem4-2}
For a real number $\lambda\not\in{\Z}$, define a sequence $\{\lambda_i\}$ by $\lambda_i=\lambda -i+1$.
Then 
$$
\omega_{\lambda}=t^{*}\wedge z^{*}+\dsum_{i=1}^{n}\lambda_i e_i^{*}\wedge f_{n-i+1}^{*}
$$
is a symplectic structure on $T^{*}\mathfrak{g}$.
\end{lem}

\begin{proof}
By the definition of $\omega_\lambda$, it is non-degenerate.

By the definition of the Lie bracket of $T^{*}\mathfrak{g}$, we have
\begin{align*}
dt^{*}&=0,\\
de_{n}^{*}&=0,\quad de_i^{*}=e_{i+1}^{*}\wedge t^{*} \quad \text{ for } i=1,\ldots, n-1,\\
df_n^{*}&=0, \quad df_i^{*}=-f_{i+1}^{*}\wedge t^{*} \quad \text{ for }i=1,\ldots ,n-1, \text{ and}\\
dz^{*}&=\dsum_{k=1}^{n-1} f_{n-k+1}^{*}\wedge e_{k+1}^{*}.
\end{align*}
Hence $\omega_{\lambda}$ satisfies
\begin{align*}
d\omega_{\lambda}
&=-\dsum_{k=1}^{n-1} t^{*}\wedge f_{n-k+1}^{*}\wedge e_{k+1}^{*}+\dsum_{i=1}^{n-1}\lambda_i e_{i+1}^{*}\wedge t^{*}\wedge f_{n-i+1}^{*} +\dsum_{i=2}^{n}\lambda_i e_i^{*}\wedge f_{n-i+2}^{*}\wedge t^{*}\\
&=\dsum_{k=1}^{n-1}(1-\lambda_k+\lambda_{k+1})t^{*}\wedge e_{k+1}^{*}\wedge f_{n-k+1}^{*}\\
&=0.
\end{align*}

\end{proof}

\begin{lem}\label{lem4-3}
The left-symmetric structure $\ddelta_{\omega_{\lambda}}$ on $T^{*}\mathfrak{g}$ induced by $\omega_{\lambda}$ coincides with the left-symmetric structure $\ddelta_{(\alpha,\beta)}$ defined by \eqref{df}, where $\alpha=\dfrac{\lambda-1}{\lambda}, \beta=-\dfrac{\lambda-n+2}{\lambda-n+1}$.
\end{lem}

We remark that $\alpha=\dfrac{\lambda-1}{\lambda}$ and $\beta=-\dfrac{\lambda-n+2}{\lambda-n+1}$ satisfy conditions \eqref{cond3-1}, \eqref{cond3-2}, and \eqref{cond3-3} since $\lambda \not\in {\Z}$.

\begin{proof}
For $\alpha=\dfrac{\lambda-1}{\lambda}, \beta=-\dfrac{\lambda-n+2}{\lambda-n+1}$, we have
\begin{align*}
\alpha_i
&=\dfrac{i(\alpha-1)+1}{(\alpha-1)(i-1)+1}=\dfrac{\lambda-i}{\lambda-i+1}=\dfrac{\lambda_{i+1}}{\lambda_{i}},  \\
\beta_i
&=\dfrac{i(\beta+1)-1}{-(\beta+1)(i+1)+1}=-\dfrac{\lambda-(n-i)+1}{\lambda-(n-i+1)+1}=-\dfrac{\lambda_{n-i}}{\lambda_{n-i+1}}, \quad \text{ and}\\
\gamma_i
&=\dfrac{\alpha_i\beta_{n-i}-\beta_{n-i}}{2\alpha_i\beta_{n-i}+\alpha_i-\beta_{n-i}}
=\dfrac{-\dfrac{\lambda_{i+1}}{\lambda_{i}}\dfrac{\lambda_i}{\lambda_{i+1}}+\dfrac{\lambda_i}{\lambda_{i+1}}}{-2\dfrac{\lambda_{i+1}}{\lambda_{i}}\dfrac{\lambda_i}{\lambda_{i+1}}+\dfrac{\lambda_{i+1}}{\lambda_{i}}+\dfrac{\lambda_i}{\lambda_{i+1}}}\\
&=\dfrac{-\lambda_{i+1}\lambda_i+\lambda_i^{2}}{\lambda_{i+1}^{2}-2\lambda_{i+1}\lambda_{i}+\lambda_i^{2}}=\lambda_i.
\end{align*}
From the above relations, we obtain
$$
-1+\alpha_{i}=-\lambda_i^{-1},\,  1+\beta_i=-\lambda_{n-i+1}^{-1},  \text{ and } -1+\gamma_{n-i}=-1+\lambda_{n-i}=\lambda_{n-i+1}.
$$

Let $\Phi_{\omega_\lambda}\colon \mathfrak{g}\to \mathfrak{g}^{*}$  be the isomorphism of vector spaces defined by $\left(\Phi_{\omega_{\lambda}}(x)\right)(y)=\omega_\lambda(x, y)$.
Then $\ddelta_{\omega_\lambda}$ is written as 
$$
x\ddelta_{\omega_\lambda} y=\Phi_{\omega_\lambda}^{-1}\circ \ad_x^{*}\circ \Phi_{\omega_\lambda}(y).
$$
Let $\ell^{\lambda}$ and $\ell^{(\alpha, \beta)}$ be left translation maps with respect to left-symmetric structures $\ddelta_{\omega_\lambda}$ and $\ddelta_{(\alpha,\beta)}$, respectively.
Since $z,e_1,$ and $ f_1$ are elements of center of $T^{*}\mathfrak{g}$, $\ell^{\lambda}_z= \ell^{\lambda}_{e_1}=\ell^{\lambda}_{f_1}=0$.

Put $e_0'=t, e_i'=e_i, e_{n+i}'=f_i,$ and $ e_{2n+1}'=z$ for $i=1,\ldots, n$.
By the definition of the Lie bracket of $T^{*}\mathfrak{g}$, we have the following representation matrices with respect to the basis $\{e_0', \ldots ,e_{2n+1}'\}$ and its dual basis: 
\begin{align*}
\ad_{t}&=\dsum_{k=1}^{n-1}E_{i, i+1}-\dsum_{k=1}^{n-1} E_{n+i, n+i+1}.\\
\ad_{e_{i+1}} &=-E_{i, 0}+E_{2n+1, 2n-i+1}\quad (i=1,\ldots , n-1).\\
\ad_{f_{i+1}}&= E_{n+i, 0} -E_{2n+1, n-i+1}\quad (i=1,\ldots , n-1).\\
\ad_{t}^{*}&=-\dsum_{k=1}^{n-1}E_{i+1, i} +\dsum_{k=1}^{n-1} E_{n+i+1, n+i}.\\
\ad_{e_{i+1}}^{*}&=E_{0, i}-E_{2n-i+1, 2n+1}\quad (i=1,\ldots , n-1).\\
\ad_{f_{i+1}}^{*}&=-E_{0, n+i}+E_{n-i+1, 2n+1}\quad (i=1,\ldots , n-1).\\
\Phi_{\omega_{\lambda}}&=E_{2n+1,0}-E_{0, 2n+1}+\dsum_{k=1}^{n}\lambda_k E_{2n-k+1, k}-\dsum_{k=1}^{n} \lambda_k E_{k, 2n-k+1}.\\
\Phi_{\omega_{\lambda}}^{-1}&=-E_{2n+1,0}+E_{0, 2n+1}-\dsum_{k=1}^{n}\lambda_k^{-1} E_{2n-k+1, k}+\dsum_{k=1}^{n} \lambda_k^{-1} E_{k, 2n-k+1}.
\end{align*}
Hence we obtain the following equations:
\begin{align*}
\ell^{\lambda}_t
&=\Phi_{\omega_{\lambda}}^{-1}\circ \ad_{t}^{*}\circ \Phi_{\omega_{\lambda}}
= \Phi_{\omega_{\lambda}}^{-1}\left(\dsum_{l=1}^{n-1}\lambda _lE_{l+1, 2n-l+1}+\dsum_{l=2}^{n}\lambda_l E_{2n-l+2, l}\right)\\
&=-\dsum_{l=1}^{n-1}\lambda_{l+1}^{-1}\lambda_{l}E_{2n-l, 2n-l+1}+\dsum_{k=1}^{n-1}\lambda_k^{-1}\lambda_{k+1}E_{k,k+1}\\
&=\dsum_{k=1}^{n-1}\left( \lambda_k^{-1}\lambda_{k+1} E_{k,k+1}-\lambda_{n-k+1}^{-1}\lambda _{n-k}E_{n+k, n+k+1}\right)\\
&=\dsum_{k=1}^{n-1}\left( \alpha_k E_{k,k+1}+\beta_kE_{n+k, n+k+1}\right)\\
&=\ell_t^{(\alpha,\beta)}.\\
\ell^{\lambda}_{e_{i+1}}
&=\Phi_{\omega_\lambda}^{-1}\circ\ad_{e_{i+1}}^{*}\circ \Phi_{\omega_\lambda}
=\Phi_{\omega_\lambda}^{-1}\left( -\lambda_i E_{0, 2n-i+1}-E_{2n-i+1, 0} \right)\\
&=\lambda_iE_{2n+1, 2n-i+1}-\lambda_i^{-1} E_{i, 0}\\
&=\gamma_iE_{2n+1, 2n-i+1}+(-1+\alpha_i) E_{i, 0}\\
&=\ell_{e_{i+1}}^{(\alpha,\beta)}\quad (i=1,\ldots, n-1).\\
\ell^{\lambda}_{f_{i+1}}
&=\Phi_{\omega_\lambda}^{-1}\circ \ad_{f_{i+1}}^{*}\circ \Phi_{\omega_\lambda}
=\Phi_{\omega_\lambda}^{-1}\left( -\lambda_{n-i+1}E_{0,n-i+1}+E_{n-i+1, 0} \right)\\
&=\lambda_{n-i+1}E_{2n+1, n-i+1}-\lambda_{n-i+1}^{-1}E_{n+i, 0}\\
&=(-1+\gamma_{n-i})E_{2n+1, n-i+1}+(1+\beta_i)E_{n+i, 0}\\
&=\ell_{f_{i+1}}^{(\alpha,\beta)}\quad (i=1,\ldots, n-1).
\end{align*}
Therefore $\ell^{\lambda}=\ell^{(\alpha,\beta)}$.
\end{proof}
\begin{thm}\label{thm4-1}
Two symplectic structures $\omega_{\lambda}$ and $\omega_{\lambda'}$ are symplectomorphic up to homothety if and only if either of the following holds:
\begin{itemize}
\item[(i)] $\lambda=\lambda'$.
\item[(ii)] $\lambda+\lambda'=n-1$, where $n$ is the dimension of $\mathfrak{g}$.
\end{itemize}
\end{thm}

\begin{proof}
Suppose that $\omega_{\lambda}$ and $\omega_{\lambda'}$ are symplectomorphic up to homothety.
Then the left-symmetric structures $\ddelta_{\omega_{\lambda}}$ and $\ddelta_{\omega_{\lambda'}}$ are isomorphic by Lemma \ref{lem4-1}.
By Lemma \ref{lem4-3} and Theorem \ref{thm3-1}, $\lambda$ and $\lambda'$ satisfiy either of the following:
\begin{itemize}
\item[(i)] $\dfrac{\lambda-1}{\lambda}=\dfrac{\lambda'-1}{\lambda'}$ and $-\dfrac{\lambda-n+1}{\lambda-n+2}=-\dfrac{\lambda'-n+1}{\lambda'-n+2}.$
\item[(ii)] $\dfrac{\lambda-i}{\lambda-i+1}=\dfrac{\lambda'-n+i+1}{\lambda'-n+i}$  and $\dfrac{\lambda-n+i+1}{\lambda-n+i}=\dfrac{\lambda'-i}{\lambda'-i+1}$ for each $i=1\ldots, {n-1}.$
\end{itemize}
If (i) holds, then we have $\lambda=\lambda'$.
If (ii) holds, then we have $\lambda+\lambda'=n-1$.

Conversely, suppose that either condition (i) or condition (ii) of Theorem \ref{thm4-1} holds.
If (i) holds, then the identity map is a symplectomorphism from $\omega_\lambda$ to $\omega_{\lambda'}$.
If (ii) holds, we define a linear isomorphism $\phi\colon T^{*}\mathfrak{g}\to T^{*}\mathfrak{g}$ by 
$$
\phi(t)=t,\,   \phi(z)=(-1)^{n} z,\, \phi(e_i)=(-1)^{n-i+1} f_i, \text{ and } \phi(f_i)=(-1)^{n-i}e_i
$$
for $i=1,\ldots ,n$.

Then $\phi$ satisfies the following equations:
\begin{align*}
\phi([t,e_{i+1}])&=\phi(e_i)=(-1)^{n-i+1} f_i. \\
[\phi(t), \phi(e_{i+1})]&=[t,(-1)^{n-i}f_{i+1}]=(-1)^{n-i+1}f_i.\\
\phi([t,f_{i+1}])&=\phi(-f_i)=(-1)^{n-i+1}e_{i}.\\
[\phi(t),\phi(f_{i+1})]&=[t,(-1)^{n-i+1}e_{i+1}]=(-1)^{n-i+1}e_{i}.\\
\phi([e_{i+1}, f_{n-i+1}])&=\phi(z)=(-1)^{n}z. \\
[\phi(e_{i+1}), \phi(f_{n-i+1})]&=[(-1)^{n-i}f_{i+1}, (-1)^{i-1}e_{n-i+1}]=(-1)^{n}z.\\
\end{align*}
Hence $\phi$ is an automorphism of $T^{*}\mathfrak{g}$.

Since $\lambda+\lambda'=n-1$, we have
$$
\lambda'_{n-l+1}=\lambda'-n+l=-\lambda+l-1=-\lambda_{l}
$$
for $l=1,\ldots, n$.
Hence we have
\begin{align*}
\phi^{*}\omega_{\lambda'}
&=(-1)^{n}t^{*}\wedge z^{*}+\dsum_{k=1}^{n}\lambda_k' (-1)^{n-k}f_k^{*}\wedge (-1)^{k}e_{n-k+1}^{*}\\
&=(-1)^{n}\left(t^{*}\wedge z^{*} -\dsum_{l=1}^{n} \lambda_{n-l+1}'e_l^{*}\wedge f_{n-l+1}^{*}\right)\\
&= (-1)^{n}\left(t^{*}\wedge z^{*} +\dsum_{l=1}^{n} \lambda_le_l^{*}\wedge f_{n-l+1}^{*}\right)\\
&=(-1)^{n}\omega_\lambda.
\end{align*}
Hence $\phi$ is a symplectomorphism up to homothety.

\end{proof}

For a fixed $n\geq 2$, any distinct elements $\lambda, \lambda'$ with 
$$
\lambda,\lambda'\in B=\left\{\lambda\in{\R} \relmiddle| \lambda>\dfrac{n-1}{2}, \lambda\not\in{\Z}\right\}
$$
do not satisfy the condition (ii) of Theorem \ref{thm4-1}.
Hence any two distinct symplectic structures in $\{\omega_\lambda\mid \lambda\in B\}$ are not symplectomorphic up to homothety.

\begin{cor}\label{cor4-1}
For any $n\geq 2$, the $n$-step nilpotent Lie algebra $T^{*}\mathfrak{g}$ admits uncountably many symplectic structures which are pairwise non-symplectomorphic up to homothety.
\end{cor}
\bibliography{reference}

\begin{bibdiv}
\begin{biblist}

\bib{Ave22}{article}{
      author={Avetisyan, Zhirayr},
       title={The structure of almost abelian lie algebras},
        date={2022},
        ISSN={1793-6519},
     journal={International Journal of Mathematics},
      volume={33},
      number={08},
       pages={Article 2250057},
         url={http://dx.doi.org/10.1142/S0129167X22500574},
}

\bib{BajBenMed07}{article}{
      author={Bajo, Ignacio},
      author={Benayadi, Sa{\"\i}d},
      author={Medina, Alberto},
       title={Symplectic structures on quadratic lie algebras},
        date={2007},
        ISSN={0021-8693},
     journal={Journal of Algebra},
      volume={316},
      number={1},
       pages={174\ndash 188},
         url={http://dx.doi.org/10.1016/j.jalgebra.2007.06.001},
}

\bib{Ben95}{article}{
      author={Benoist, Yves},
       title={A non-affine nilvariety},
    language={French},
        date={1995},
        ISSN={0022-040X},
     journal={J. Differ. Geom.},
      volume={41},
      number={1},
       pages={21\ndash 52},
}

\bib{BurGru95}{article}{
      author={Burde, D.},
      author={Grunewald, F.},
       title={Modules for certain {Lie} algebras of maximal class},
    language={English},
        date={1995},
        ISSN={0022-4049},
     journal={J. Pure Appl. Algebra},
      volume={99},
      number={3},
       pages={239\ndash 254},
}

\bib{Bur98}{article}{
      author={Burde, Dietrich},
       title={Simple left-symmetric algebras with solvable {Lie} algebra},
    language={English},
        date={1998},
        ISSN={0025-2611},
     journal={Manuscr. Math.},
      volume={95},
      number={3},
       pages={397\ndash 411},
}

\bib{Bur06}{article}{
      author={Burde, Dietrich},
       title={Left-symmetric algebras, or pre-{Lie} algebras in geometry and
  physics},
    language={English},
        date={2006},
        ISSN={1895-1074},
     journal={Cent. Eur. J. Math.},
      volume={4},
      number={3},
       pages={323\ndash 357},
}

\bib{Cas22}{article}{
      author={Castellanos~Moscoso, Luis~Pedro},
       title={Left-invariant symplectic structures on diagonal almost abelian
  {Lie} groups},
        date={2022},
        ISSN={0018-2079},
     journal={Hiroshima Mathematical Journal},
      volume={52},
      number={3},
       pages={357\ndash 378},
         url={http://dx.doi.org/10.32917/h2021055},
}

\bib{CasTam23}{article}{
      author={Castellanos~Moscoso, Luis~Pedro},
      author={Tamaru, Hiroshi},
       title={A classification of left-invariant symplectic structures on some
  {Lie} groups},
    language={English},
        date={2023},
        ISSN={0138-4821},
     journal={Beitr. Algebra Geom.},
      volume={64},
      number={2},
       pages={471\ndash 491},
}

\bib{Chu74}{article}{
      author={Chu, Bon-Yao},
       title={Symplectic homogeneous spaces},
    language={English},
        date={1974},
        ISSN={0002-9947},
     journal={Trans. Am. Math. Soc.},
      volume={197},
       pages={145\ndash 159},
}

\bib{DekOng00}{article}{
      author={Dekimpe, Karel},
      author={Ongenae, Veerle},
       title={On the number of abelian left symmetric algebras},
    language={English},
        date={2000},
        ISSN={0002-9939},
     journal={Proc. Am. Math. Soc.},
      volume={128},
      number={11},
       pages={3191\ndash 3200},
}

\bib{BouMan24}{article}{
      author={El~Bourkadi, Said},
      author={Mansouri, Mohammed~W.},
       title={Left-symmetric products on cosymplectic {Lie} algebras},
        date={2024},
     journal={J. Lie Theory},
      volume={34},
      number={2},
       pages={249\ndash 265},
}

\bib{GozRem03}{article}{
      author={Goze, Michel},
      author={Remm, Elisabeth},
       title={Affine structures on abelian {Lie} groups},
    language={English},
        date={2003},
        ISSN={0024-3795},
     journal={Linear Algebra Appl.},
      volume={360},
       pages={215\ndash 230},
}

\bib{Hel79}{article}{
      author={Helmstetter, Jacques},
       title={Radical d'une alg{\`e}bre sym{\'e}trique {\`a} gauche},
    language={French},
        date={1979},
        ISSN={0373-0956},
     journal={Ann. Inst. Fourier},
      volume={29},
      number={4},
       pages={17\ndash 35},
}

\bib{Kat24}{article}{
      author={Kato, Naoki},
       title={Double extensions of left-symmetric structures},
        date={2024},
        ISSN={1532-4125},
     journal={Communications in Algebra},
      volume={52},
      number={10},
       pages={4180\ndash 4188},
         url={http://dx.doi.org/10.1080/00927872.2024.2343753},
}

\bib{KhaGozMed04}{article}{
      author={Khakimdjanov, Yu.},
      author={Goze, M.},
      author={Medina, A.},
       title={Symplectic or contact structures on {Lie} groups},
        date={2004},
        ISSN={0926-2245},
     journal={Differential Geometry and its Applications},
      volume={21},
      number={1},
       pages={41\ndash 54},
         url={http://dx.doi.org/10.1016/j.difgeo.2003.12.006},
}

\bib{Kim86}{article}{
      author={Kim, Hyuk},
       title={Complete left-invariant affine structures on nilpotent {Lie}
  groups},
    language={English},
        date={1986},
        ISSN={0022-040X},
     journal={J. Differ. Geom.},
      volume={24},
       pages={373\ndash 394},
}

\bib{Kui53}{article}{
      author={Kuiper, N.~H.},
       title={Sur les surfaces localement affines},
         how={Colloques internat. {Centre} nat. {Rech}. {Sci}. 52, 79-87
  (1953).},
        date={1953},
     journal={Colloques internat. {Centre} nat. {Rech}. {Sci}.},
      volume={52},
       pages={79\ndash 87},
}

\bib{MedRev91}{inproceedings}{
      author={Medina, Alberto},
      author={Revoy, Philippe},
       title={Groupes de {Lie} {\`a} structure symplectique invariante},
        date={1991},
   booktitle={{Symplectic Geometry, Groupoids, and Integrable Systems,
  S{\'e}minaire Sud Rhodanien de G{\'e}om{\'e}trie {\`a} Berkeley}},
      editor={Dazord, Pierre},
      editor={Weinstein, Alan},
      series={Mathematical Sciences Research Institute Publications},
   publisher={Springer New York},
       pages={247\ndash 266},
}

\bib{Mil77}{article}{
      author={Milnor, John~W.},
       title={On fundamental groups of complete affinely flat manifolds},
    language={English},
        date={1977},
        ISSN={0001-8708},
     journal={Adv. Math.},
      volume={25},
       pages={178\ndash 187},
}

\bib{NagYag74}{article}{
      author={Nagano, Tadashi},
      author={Yagi, Katsumi},
       title={The affine structures on the real two-torus. {I}},
    language={English},
        date={1974},
        ISSN={0030-6126},
     journal={Osaka J. Math.},
      volume={11},
       pages={181\ndash 210},
}

\bib{Ova06}{article}{
      author={Ovando, Gabriela},
       title={Four dimensional symplectic {Lie} algebras},
    language={English},
        date={2006},
        ISSN={0138-4821},
     journal={Beitr. Algebra Geom.},
      volume={47},
      number={2},
       pages={419\ndash 434},
}

\bib{Ova16}{article}{
      author={Ovando, Gabriela~P.},
       title={Lie algebras with ad-invariant metrics. {A} survey - guide},
        date={2016},
     journal={Rend. Semin. Mat., Univ. Politec. Torino},
      volume={74},
      number={1},
       pages={243\ndash 268},
}

\bib{Sch74}{article}{
      author={Scheuneman, John},
       title={Affine structures on three-step nilpotent {Lie} algebras},
        date={1974},
        ISSN={1088-6826},
     journal={Proceedings of the American Mathematical Society},
      volume={46},
      number={3},
       pages={451\ndash 454},
         url={http://dx.doi.org/10.1090/S0002-9939-1974-0412344-2},
}

\end{biblist}
\end{bibdiv}
\par\noindent{\scshape \small
Faculty of Liberal Arts and Sciences,
Chukyo University, \\
101 Tokodachi, Kaizu-cho, Toyota-shi, Aichi 470-0393, Japan.}
\par\noindent{knaoki@lets.chukyo-u.ac.jp}

\end{document}